\documentclass[reqno]{amsart}
\usepackage{amsmath}%
\usepackage{amsfonts}%
\usepackage{amssymb}%
\usepackage{graphicx}
\usepackage{mathrsfs}
\usepackage{hyperref}
\usepackage{fullpage}
\usepackage{cite}
\def\0{\emptyset}

\def\n{\noindent}

\newtheorem{theorem}{Theorem}
\newtheorem{lemma}[theorem]{Lemma}
\newtheorem{corollary}[theorem]{Corollary}
\newtheorem{claim}[theorem]{Claim}

\newtheorem{definition}[theorem]{Definition}
\def\eps{\varepsilon}

\begin{document}
\title{Vertex degree sums for matchings in 3-uniform hypergraphs }
\thanks{
Yi Zhang and Mei Lu are supported by the National Natural Science Foundation of China (Grant 11771247).
Yi Zhao is partially supported by NSF grant DMS-1700622.}

\author{Yi Zhang}
\address{ School of Science, Beijing University of Posts and Telecommunications, Beijing, 100876}
\email{shouwangmm@sina.com}

\author{Yi Zhao}
\address
{Department of Mathematics and Statistics, Georgia State University, Atlanta, GA 30303}
\email{yzhao6@gsu.edu}

\author{Mei Lu}
\address{Department of Mathematical
Sciences, Tsinghua University, Beijing, 100084}
\email{lumei@tsinghua.edu.cn}

\date{\today}


\begin{abstract}
Let $n, s$ be positive integers such that $n$ is sufficiently large and $s\le n/3$. Suppose $H$ is a 3-uniform hypergraph of order $n$. If $H$ contains no isolated vertex and $\deg(u)+\deg(v) > 2(s-1)(n-1)$ for any two vertices $u$ and $v$ that are contained in some edge of $H$, then $H$ contains a matching of size $s$. This degree sum condition is best possible and confirms a conjecture of the authors [Electron. J. Combin. 25 (3), 2018], who proved the case when $s= n/3$.
\end{abstract}

\maketitle

\section{Introduction}

A \emph{$k$-uniform hypergraph} $H$ (in short, \emph{$k$-graph}) is a pair $(V,E)$, where $V$ is a finite set
of vertices and $E$ is a family of $k$-element subsets of $V$. Note that a 2-graph is simply a graph.
Let $V(H)$ and $E(H)$ denote the vertex set and edge set of $H$, respectively.
A \emph{matching of size $s$} in $H$ is a family of $s$ pairwise disjoint edges of $H$. If the matching covers all the vertices of $H$, then we call it a \emph{perfect matching}.
Given a set $S \subseteq V$, the \emph{degree} $\deg_{H}(S)$ of $S$ is the number of the edges of $H$ containing $S$. We simply write $\deg(S)$ when $H$ is obvious from the context. Further, let $\delta_\ell(H)= \min\{\deg(S): S\subseteq V(H),|S|=\ell\}$.

Given integers $\ell<k\le n$ such that $k$ divides $n$, let $m_\ell(k,n)$ denote the smallest integer $m$ such that every $k$-graph $H$ on $n$ vertices with $\delta_\ell(H) \geq m$ contains a perfect matching. In recent years the problem of determining $m_\ell(k,n)$ has received much attention (see \cite{Alon,Han,Han3,Kha1,Kha2,Kuhn1,Kuhn2,Mar,Pik,Rod1,Rod2,Rod3,TrZh12, TrZh13, TrZh16}). In particular, R\"{o}dl,  Ruci\'{n}ski and Szemer\'{e}di \cite{Rod3} determined $m_{k-1}(k, n)$ for all $k\ge 3$ and sufficiently large $n$.
Treglown and Zhao \cite{TrZh12, TrZh13} determined $m_\ell(k,n)$ for all $\ell \geq k/2$ and sufficiently large $n$.
More Dirac-type results on hypergraphs can be found in surveys \cite{RoRu-s, Zhao}.

A well-known result of Ore \cite{Ore} extended Dirac's theorem by determining the smallest degree sum of two non-adjacent vertices that guarantees a Hamilton cycle in graphs. Ore-type problems for hypergraphs have been studied recently. For example, Tang and Yan \cite{Tang} studied the degree sum of two $(k-1)$-sets that guarantees a tight Hamilton cycle in $k$-graphs. Zhang and Lu \cite{Yi1} studied the degree sum of two $(k-1)$-sets that guarantees a perfect matching in $k$-graphs. Zhang, Zhao and Lu \cite{zhang} determined the minimum degree sum of two adjacent vertices that guarantees a perfect matching in 3-graphs without isolated vertices, see Theorem~\ref{theoremb1} (two vertices in a hypergraph are \emph{adjacent} if  there exists an edge containing both of them). Note that one may study the minimum degree sum of two arbitrary vertices and that of two non-adjacent vertices that guarantees a perfect matching instead. In fact, it was mentioned in \cite{zhang} that the former equals to $2 m_1(3, n)-1$ while the latter does not exist.

Let us define (potential) extremal 3-graphs for the matching problem. For $1\le \ell\le 3$, let $H^{\ell}_{n, s}$ denote the 3-graph of order $n$, whose vertex set is partitioned into two sets $S$ and $T$ of size $n- s\ell+1$ and $s\ell -1$, respectively, and whose edge set consists of all triples with at least $\ell$ vertices in $T$. A well-known conjecture of Erd\H{o}s ~\cite{Erd65}, recently verified for $3$-graphs \cite{Fra, LuMi} , implies that $H_{n,s}^1$  or $H_{n,s}^3$  is the densest $3$-graph on $n$ vertices not containing a matching of size $s$. On the other hand, K\"{u}hn, Osthus and Treglown \cite{Kuhn2} showed that for sufficiently large $n$, $H_{n,s}^1$ has the largest minimum vertex degree among all 3-graphs on $n$ vertices not containing a matching of size $s$.

\begin{theorem}\cite{Kuhn2}\label{Kuhn2}
There exists $n_0 \in \mathbb{N}$ such that if $H$ is a $3$-graph of order $n \geq n_0$ with $\delta_1(H) > \delta_1(H^1_{n, s})= \binom{n-1}{2}-\binom{n-s}{2}$ and $n \geq 3s$, then $H$ contains a matching of size $s$.
\end{theorem}

Given a 3-graph $H$, let $\sigma_2(H)$ denote the minimum $\deg(u)+\deg(v)$ among all adjacent vertices $u$ and $v$. It is easy to see that
\begin{align*}
\sigma_2(H_{n,s}^3) &= 2 \binom{3s-2}{2}, \quad
\sigma_2(H_{n,s}^1) =2\left(\binom{n-1}{2}-\binom{n-s}{2} \right), \ \text{and} \\
\sigma_2(H_{n,s}^2) & =\binom{2s-2}{2}+\left(n-2s+1\right)\binom{2s-2}{1}+\binom{2s-1}{2} = (2s-2)(n-1).
\end{align*}

The following is \cite[Theorem 1]{zhang}, which implies that, when $n$ is divisible by 3 and sufficiently large, $H^2_{n, n/3}$ has the largest $\sigma_2(H)$ among all $n$-vertex 3-graphs $H$ containing no isolated vertex or perfect matching.

\begin{theorem}\label{theoremb1}\cite{zhang}
There exists $n_0 \in \mathbb{N}$ such that the following holds for all integers $n\ge n_0$ that are divisible by $3$. Let $H$ be a $3$-graph of order $n$ without isolated vertex. If $\sigma_2(H) > \sigma_2(H^2_{n,n/3})= \frac{2}{3}n^2-\frac{8}{3}n+2$, then $H$ contains a perfect matching.
\end{theorem}

Zhang, Zhao and Lu \cite[Conjecture 12]{zhang} further conjectured that for sufficiently large $n$ and any $s< n/3$, $H^2_{n, s}$ has the largest $\sigma_2(H)$ among all $n$-vertex 3-graphs $H$ containing no isolated vertex or matching of size $s$. In this paper we verify this conjecture.

\begin{theorem}
\label{the1}
There exists $n_1 \in \mathbb{N}$ such that the following holds for all integers $n\ge n_1$ and $s\le n/3$. If $H$ is a $3$-graph of order $n$ without isolated vertex and $\sigma_2(H) > \sigma_2(H_{n,s}^2)= 2(s-1)(n-1)$, then $H$ contains a matching of size $s$.
\end{theorem}

Since two theorems have different extremal hypergraphs, Theorem~\ref{the1} does not imply Theorem~\ref{Kuhn2} (analogously Theorem~\ref{Kuhn2} does not imply Erd\H os' matching conjecture for 3-graphs). On the other hand, one may wonder why we assume that $H$ contains no isolated vertex in Theorem~\ref{the1} (especially when  $s< n/3$). In fact, as shown in the concluding remarks of \cite{zhang}, Theorem~\ref{the1} implies another conjecture \cite[Conjecture 13]{zhang}, which determines the largest $\sigma_2(H)$ among all 3-graphs containing no matching of size $s$. Note that  $\sigma_2(H_{n,s}^2) \geq \sigma_2(H_{n,s}^3)$ if and only if $s \le (2n+4)/9$.

\begin{corollary}
There exists $n_2 \in \mathbb{N}$ such that the following holds. Suppose that $H$ is a $3$-graph of order $n \geq n_2$ and $2\le s\le n/3$. If $\sigma_2(H) > \sigma_2(H_{n,s}^2)$ and $s \leq (2n+4)/9$ or $\sigma_2(H) > \sigma_2(H_{n,s}^3)$ and $s > (2n+4)/9$, then $H$ contains a matching of size $s$.
\end{corollary}

Let us explain our approach towards Theorem~\ref{the1}. The case when $s\le n/13$ was already solved by Zhang and Lu \cite{Yi2} in a stronger form. Note that $\sigma_2(H_{n,s}^2) > \sigma_2(H_{n,s}^1)$.
The following theorem shows that, when $n\ge 13s$, not only $H_{n,s}^2$ is the (unique) 3-graph with the largest $\sigma_2(H)$ among all $H$ containing no isolated vertex or a matching of size $s$, but also $H^1_{n, s}$ is the sub-extremal 3-graph for this problem. (In fact, Zhang and Lu~\cite{Yi2} conjectured that Theorem~\ref{Yi1} holds for \emph{all} $n\ge 3s$. If true, this strengthens Theorem~\ref{Kuhn2} and actually provides a link between Ore's and Dirac's problems.)

\begin{theorem}\cite{Yi3}\label{Yi1}
Let $n, s$ be positive integers and $H$ be a 3-graph of order $n \geq 13s$ without isolated vertex. If $\sigma_2(H)  > \sigma_2(H^1_{n, s}) = 2\left( \binom{n-1}{2} - \binom{n-s}{2} \right)$, then either $H$ contains a matching of size $s$ or $H$ is a subgraph of $H_{n,s}^2$.
\end{theorem}

Therefore it suffices to prove Theorem~\ref{the1} for reasonably large $s$. For such $s$, we actually prove a (stronger) stability theorem.

\begin{theorem}\label{the2}
Given $0 < \varepsilon \ll \tau \ll 1$, let $n$ be sufficiently large and $ \tau n < s \leq n/3$.
If $H$ is a $3$-graph of order $n$ without isolated vertex such that $\sigma_2(H) > 2sn -\varepsilon n^2$,  then either $H$ is a subgraph of $H_{n,s}^2$ or $H$ contains a matching of size $s$.
\end{theorem}

Theorem~\ref{the1} follows from Theorem~\ref{the2} immediately. Indeed, if $\sigma_2(H) > \sigma_2(H^2_{n, s})$, then it is easy to see that $H$ is not a subgraph of $H_{n,s}^2$.\footnote{Unfortunately $\sigma_2$ is not a monotone function: for example, adding an edge to $H^2_{n, s}$ indeed reduces the value of $\sigma_2$ because two vertices in $S$ now become adjacent and their degree sum is smaller than $\sigma_2(H^2_{n, s})$.} Suppose instead, that $V(H)$ can be partitioned $S\cup T$ such that $|S|= n- 2s+1$, $|T|= 2s-1$, and every edge of $H$ contains at least two vertices of $T$. Since $H$ contains no isolated vertices, every vertex of $S$ is adjacent to some vertex of $T$. Thus $\sigma_2(H)\le \deg(u) + \deg(v)$ for some $u\in S$ and $v\in T$. Consequently $\sigma_2(H) \le \sigma_2(H^2_{n, s})$, a contradiction.
We therefore apply Theorem \ref{the2} to derive that $H$ contains a matching of size $s$. Furthermore, Theorem \ref{the2} implies that $H_{n,s}^2$ is the unique extremal 3-graph for Theorem \ref{the1} because all proper subgraphs $H$ of $H_{n,s}^2$  satisfy $\sigma_2(H) < \sigma_2(H_{n,s}^2)$.

In order to prove Theorem~\ref{the2}, we follow the same approach as in \cite{zhang}: using the condition on $\sigma_2(H)$, we greedily extend a matching of $H$ until it has $s$ edges. An important intermediate step is finding a matching that covers a certain number of low-degree vertices (see Lemma~\ref{lemma7}). Nevertheless, the proof of Theorem~\ref{the2} does require new ideas: in particular, the meaning of an \emph{optimal} matching is more complicated (see Definition~\ref{def:M}); we proceed differently depending on whether the number of low-degree vertices in the optimal matching is at the threshold. In one case we reduce the problem to that of finding a perfect matching in a subgraph of $H$ and apply the main result of  \cite{zhang} (see Theorem~\ref{theoremb6}).


This paper is organized as follows. In Section 2, we give an outline of the proof along with some preliminary results. We prove Lemma~\ref{lemma7} in Section 3 and complete the proof in Section 4.

\vskip.2cm

\n{\bf Notation}
Given a graph $G$ and a vertex $u$ in $G$, $N_G(u)$ is the set of neighbors of $u$ in $G$.
Suppose $H$ is a 3-uniform hypergraph. 
For $u\ne v\in V(H)$, let $N_H(u,v)=\{w\in V(H):\{u,v,w\} \in E(H)\}$ (the subscript is often omitted when $H$ is clear from the context). Given three subsets $V_1,V_2,V_3$ of $V(H)$, we say that an edge $\{v_1,v_2,v_3\}\in E(H)$ is a type of $V_1V_2V_3$ if $v_i \in V_i$ for $1\le i\le 3$.
Given a vertex $v\in V(H)$ and a subset $A \subseteq V(H)$, we define \emph{the link} $L_{v}(A)=\{uw:u, w\in A\mbox{~and~}\{u,v,w\}\in E(H)\}$. When $A$ and $B$ are two disjoint subsets of $V(H)$, we let $L_{v}(A,B)=\{uw:u\in A,~ w\in B \mbox{~and~}\{u,v,w\}\in E(H)\}$.

We write $0 < a_1 \ll a_2 \ll a_3 $ if we can choose the constants $a_1, a_2, a_3$ from right to left. More precisely there are increasing functions $f$ and $g$ such that given $a_3$, whenever we choose some $a_2\leq f(a_3)$ and $a_1 \leq g(a_2)$, all calculations needed in our proof are valid.
\vskip.2cm

\section{Outline of the proof and preliminaries}

Let $n$ be sufficiently large and $ \tau n < s \leq n/3$. Suppose $H$ is a $3$-graph of order $n$ without isolated vertex and $\sigma_2(H) > 2sn -\varepsilon n^2$. Let $U = \{u\in V(H): \deg(u) > sn- \frac{\varepsilon}2 n^2\}$ and $W = V\setminus U$.  Then any two vertices of $W$ are not adjacent -- otherwise $\sigma_2(H)\le  2sn-\varepsilon n^2$, a contradiction. If $|U| \leq 2s-1$, then $H$ is a subgraph of $H_{n,s}^2$ and we are done. We thus assume that $|U| \geq 2s$.

Throughout the proof we use small constants
\begin{align}\label{equ1}
0 <  \varepsilon  \ll \varepsilon'  \ll  \varepsilon''  \ll \eta_1 \ll \eta_2  \ll \gamma \ll \gamma'\ll \tau \ll 1.
\end{align}

We first prove the following lemma, which is an extension of \cite[Lemma 4]{zhang}.

\begin{lemma} \label{lemma7}
Given $0 < \varepsilon \ll \tau \ll 1$, let $n$ be sufficiently large and $ \tau n < s \leq n/3$.
Suppose $H$ is a $3$-graph of order $n$ without isolated vertex and $\sigma_2(H) > 2sn -\varepsilon n^2$.
Let $U =\{u \in V(H): \deg(u) > sn-\varepsilon n^2/2\}$ and $W=V\setminus U$. If $2s \leq |U| \leq 3s$, then $H$ contains a matching of size $3s-|U|$, each of which contains exactly one vertex of $W$.
\end{lemma}

\begin{definition}\label{def:M}
We call a matching $M$ \emph{optimal} if  (i) $M$ contains a submatching $M_1 = \{e\in M : e\cap W \neq \emptyset  \}$ of size at least  $3s-|U|$; (ii) subject to (i), $|M|$ is as large as possible; (iii) subject to (i) and (ii), $|M_1|$ is as large as possible.
\end{definition}


Lemma \ref{lemma7} shows that $H$ contains an optimal matching $M$. We separate the cases when $|M_1| = 3s-|U|$ and when $|M_1| > 3s-|U|$. When $|M_1| = 3s-|U|$, we first consider the case when $s \leq n/3-\eta_1n$. If no vertex of $U_3:= U\setminus V(M)$ is adjacent to any vertex of $W_2:= W\setminus V(M)$, then the assumption  $|M_1| = 3s-|U|$ forces $\sum_{i=1}^{3}\deg(u_i)$ to be smaller than $3sn-\frac{3}{2}\varepsilon n^2$ for any three vertices $u_1, u_2, u_3\in U_3$. If some vertex $u_1\in U_3$ is adjacent to $v_1\in W_2$, then the fact  $v_1 \in W$ reduces $ \sum_{i=1}^2\deg(u_i)+\deg(v_1)$ to a number less than $3sn-\frac{3}{2}\varepsilon n^2$ (where $u_2$ is another vertex of $U_3$).  When $s > n/3-\eta_1n$, we consider $H' =H[V\setminus W_2]$. Since $|W_2|= n-3s$ is very small, we deduce that $\sigma_2(H')$ is greater than $2sn- \eta_2 n^2$. This allows us to apply the following theorem from \cite{zhang} to obtain a perfect matching of $H'$, which is also a matching of size $s$ of $H$.

\begin{theorem} \label{theoremb6}\cite{zhang}
There exist $\eta_2>0$ and $n_0 \in \mathbb{N}$ such that the following holds for all integers $n\ge n_0$ that are divisible by $3$.  Suppose that $H$ is a $3$-graph of order $n$ without isolated vertex and $\sigma_2(H) > 2n^2/3-\eta_2 n^2$, then either $H$ is a subgraph of $ H^2_{n,n/3}$ or $H$ contains a perfect matching.
\end{theorem}

Now consider the case when $|M_1| > 3s-|U|$. Let $W':=\{v \in W: \deg(v) \leq sn- s^2/2+\gamma' n^2\}$. If $|W'|$ is very small,  then we can find a matching of size $s$ in $H[V\setminus W']$ by Theorem \ref{Kuhn2}.  When $|W'|$ is not small, we consider $u_1, u_2, u_3\in U_3$. If one of $u_1$, $u_2$, $u_3$ is adjacent to one vertex from $W'$, then $\sum_{i=1}^{3}\deg(u_i)$ becomes much larger than $3sn$; otherwise we show that $\sum_{i=1}^{3}\deg(u_i)< 3sn-\frac{3}{2}\varepsilon n^2$ by proceeding with the cases when $|W' \cap W_1|  > \gamma n/2$ and when $|W' \cap W_2|  > \gamma n/2$ separately.


\medskip

\medskip
In the proof we need several (simple) extremal results on (hyper)graphs. Lemma \ref{lemma1} is Observation 1.8 of Aharoni and Howard \cite{Aha}. Lemmas \ref{lemmaa2} and \ref{lemma3} are from \cite{zhang}.
A $k$-graph $H$ is called \emph{$k$-partite} if $V(H)$ can be partitioned into $V_1,\ldots,V_k$, such that each edge of $H$ meets every $V_i$ in precisely one vertex. If all parts are of the same size  $n$, we call $H$ \emph{$n$-balanced}.

\begin{lemma}\cite{Aha}
\label{lemma1}
Let $F$ be the edge set of an $n$-balanced $k$-partite $k$-graph. If $F$ does not contain $s$ disjoint edges, then $|F| \leq (s-1)n^{k-1}$.
\end{lemma}

\begin{lemma}\cite{zhang}\label{lemmaa2}
Let $G_1, G_2, G_3$ be three graphs on the same set $V$ of $n\ge 4$ vertices such that every edge of $G_1$ intersects every edge of $G_i$ for both $i=2, 3$. Then $\sum_{i=1}^3\sum_{v\in A} \deg_{G_i}(v) \leq 6(n-1)$ for any set $A\subset V$ of size $3$.
\end{lemma}

\begin{lemma}\cite{zhang}\label{lemma3}
Let $G_1, G_2, G_3$ be three graphs on the same set $V$ of $n\ge 5$ vertices such that for any $i\ne j$, every edge of $G_i$ intersects every edge of $G_j$.
Then $\sum_{i=1}^3\sum_{v\in A} \deg_{G_i}(v) \leq 3(n+1)$ for any set $A\subset V$ of size $3$.
\end{lemma}

Following the same proof of Lemmas \ref{lemmaa2} and \ref{lemma3} from \cite{zhang}, we obtain another lemma and omit its proof.

\begin{lemma}\label{lemmaa3}
Let $G_1, \cdots,G_k$ be $k$ graphs on the same set $V$ of $n \geq 4$ vertices such that for any $1 \leq i < j \leq k$, every edge of  $G_i$ intersects  every edge of $G_j$. Then $\sum_{i=1}^k\sum_{v\in A} \deg_{G_i}(v) \leq kn$ for any set $A\subset V$ of size $2$.\qed
\end{lemma}

The following lemma needs slightly more work so we include a proof.
\begin{lemma}\label{Lemma4}
Given two disjoint vertex sets $A=\{u_1,u_2,\ldots,u_a\}$ and $B=\{v_1,v_2,\ldots,v_b\}$ with $ a \geq 3$ and $ b\geq 1$. Let $G_i$, $i=1,2,3$, be graphs on $A\cup B$ such that every vertex of $B$ is isolated vertex in $G_1$, and every edge of $G_i$ ($i=2,3$) contains at least one vertex of $A$. If there are no two disjoint edges (i) one from $G_1$ and the other from $G_2$ or $G_3$; or (ii) one from $G_2$ and the other from $G_3$, and at least one of them contains a vertex from $B$, then
 \[
 \sum_{i=1}^3\left(\sum_{j=1}^2 \deg_{G_i}(u_j)+ \deg_{G_i}(v_1)\right) \leq \max\{4a+7,3a+2b+5\}.\]
\end{lemma}
\begin{proof}
For convenience, let $s_i = \sum_{j=1}^2 \deg_{G_i}(u_j)+ \deg_{G_i}(v_1)$ for $i=1, 2, 3$ and $y= s_1 + s_2 + s_3$.
Below we show that $y\le \max\{4a+7,3a+2b+5\}$.

We first observe that if $\deg_{G_i}(v_1) \ge 3$ for some $i\in \{2, 3\}$, then $E(G_1)= \emptyset$ and $G_{i'}$ is a star centered at $v_1$, where $i'= 5- i$. Indeed, if $G_1$ or $G_{i'}$ contains an edge $e$ not incident to $v_1$, then $e$ is disjoint from some edge of $G_i$ that is incident to $v_1$ -- this contradicts our assumption. The observation implies that if $\deg_{G_i}(v_1)\ge 3$ for both $i=2, 3$, then $E(G_1)= \emptyset$ and both $G_2$ and $G_3$ are stars centered at $v_1$. In this case, $s_i\le a+2$ for $i=2, 3$
and thus $y\le 2(a+2)$. If $\deg_{G_2}(v_1)\ge 3$ and $\deg_{G_3}(v_1)\le 2$, then $E(G_1)= \emptyset$ and $G_3$ consists of at most two edges incident to $v_1$. In this case, $y \le 2(a+b-1)+ a + 4= 3a + 2b + 2$. The case when $\deg_{G_2}(v_1)\le 2$ and $\deg_{G_3}(v_1)\ge 3$ is analogous.
We thus assume that
\begin{align}\label{eq:v1}
\deg_{G_i}(v_1) \le 2 \quad \text{for } i=2, 3
\end{align}
for the rest of the proof.

Next, we observe that if $|N_{G_i}(u_j) \cap B| \ge 2 $ for some $i\in \{2, 3\}$ and some $j\in \{1, 2\}$, then $G_{i'}$ is a star centered at $u_j$ for $i'\in \{1, 2, 3\}\setminus \{i\}$. This is again due to our assumption on $G_1, G_2$ and $G_3$.
The observation implies that if $|N_{G_i}(u_j) \cap B| \ge 2 $ for both $j=1, 2$, then $E(G_{i'})\subseteq \{ u_1 u_2 \}$ and consequently, $s_{i'}\le 2$ for $i'\in \{1, 2, 3\}\setminus \{i\}$. By \eqref{eq:v1}, we have $s_i\le 2(a+b-1)+2$. Therefore, $y\le 2(a+b-1)+2+4= 2a+ 2b+4$. The observation also implies that if $|N_{G_i}(u_j) \cap B| \ge 2 $ for both $i=2, 3$, then $G_1, G_2, G_3$ are all stars centered at $u_j$. In this case, $s_1\le a$ and $s_i\le a+b+1$ for $i=2, 3$, which implies that $y\leq a+ 2(a+b+1) = 3a + 2b+2$. We now consider the case when $|N_{G_2}(u_1) \cap B| \ge 2 $, $|N_{G_2}(u_2) \cap B| \le 1 $, and $|N_{G_3}(u_1) \cap B| \le 1$. Thus $G_3$ is a star (centered at $u_1$) of size at most $a$, which yields $s_3\le a+2$. Now suppose $N_{G_2}(u_2) \cap B \subseteq \{v_p\}$ for some $p$. Let $A':= A\cup \{v_p\}$ (note that $|A'|= a+1\ge 4$). Since every edge of $G_1$ intersects every edge of $G_2$, we can apply Lemma~\ref{lemmaa3} to $G_1[A']$ and $G_2[A']$ and obtain that $\sum_{i=1}^2 \sum_{j=1}^2 \deg_{G_i[A']}(u_j) \le 2a+2$. Since $| N_{G_2}(u_1)\cap (B\setminus \{v_p\})|\le b-1$ and $\deg_{G_2}(v_1)\le 2$, it follows that $s_1 + s_2 \le 2a + 2 + b-1 + 2 = 2a + b+3$ and $y\le 2a + b + 3 + a+2 = 3a + b + 5$.

We thus assume that $|N_{G_i}(u_j) \cap B| \le 1 $ for $i=2, 3$ and $j=1, 2$. Suppose $N_{G_2}(u_2) \cap B \subseteq \{v_p\}$ for some $p$ and let $A':= A\cup \{v_p\}$. We apply Lemma~\ref{lemmaa3} to $G_1[A']$ and $G_2[A']$ and obtain that $\sum_{i=1}^2 \sum_{j=1}^2 \deg_{G_i[A']}(u_j) \le 2a+2$. Since $| N_{G_2}(u_1)\cap B|\le 1$ and $\deg_{G_2}(v_1)\le 2$, it follows that $s_1 + s_2 \le 2a + 2 + 1 + 2$. On the other hand, we have $s_3\le 2a+2$ because $\deg_{G_3}(u_j)\le a$ for $j=1, 2$ and $\deg_{G_3}(v_1)\le 2$. Thus $y\leq 2a+ 5 + 2a + 2 = 4a + 7$.
\end{proof}

\section{Proof of Lemma \ref{lemma7}}
The proof is similar to that of \cite[Lemma 4]{zhang}. 
Let $M$ be a largest matching of $H$ such that each edge of $M$ contains (exactly) one vertex of $W$. To the contrary, assume $|M| \leq 3s-|U|-1$.  Let $U_1 = V(M) \cap U$, $U_2 = U\setminus U_1$, $W_1 = V(M) \cap W$ and $W_2 =W\setminus W_1$. Since $|U|\ge 2s$, we have $|U_2|=|U|-2|M|\ge 2$. Since $|W_2|= |W| - |M|$ and $|W|\ge 3s - |U|$, it follows that $W_2\ne \emptyset$.

Below is a sketch of the proof.
We first assume $|U|< 2s + \eps' n$. In this case every vertex in $U$ is adjacent to some vertex in $W$.
If $|M|$ is not close to $s$, then we easily obtain a contradiction because $U_2$ is not small. When $|M|$ is close to $s$, we consider three vertices $u_1\ne u_2\in U_2$ and $v_0\in W_2$,  and derive a contradiction on $\deg(u_1) + \deg(u_2) + \deg(v_0)$. Next we assume $|U|\ge 2s + \eps' n$. In this case $U_2$ is not small. If no vertex of $W_2$ is adjacent to any vertex of $U_2$, then consider two adjacent vertices $v_0\in W_2$ and $u_0\in U_1$. We have $\deg(v_0)\le \binom{2|M|}{2}$, which eventually yields that $\deg(v_0) + \deg(u_0) < 2sn - \eps n^2$.
Now assume $v_0\in W_2$ is adjacent to some vertex $u_0\in U_2$. In this case we define $M'$ consisting of all $e \in M$ that contains a vertex $u'\in U$ such that $|N(v_0, u')\cap U_2|\ge 3$. We show that if $|M'|$ is small, then $\deg(v_0)$ is small; otherwise $\deg(u_0)$ is small. In either case we derive that $\deg(v_0) + \deg(u_0) < 2sn - \eps n^2$.

We now give the details of the proof.

\noindent{\bf Case 1. } $2s \leq |U|<2s+\varepsilon'n$.

In this case we have the following two claims.

\begin{claim}\label{claim16}
$|M| \geq s-\varepsilon'' n$.
\end{claim}
\begin{proof} To the contrary, assume that $|M| < s-\varepsilon'' n$. Fix $v_0\in W_2$. Then $\deg(v_0) \leq \binom{|U|}{2} - \binom{|U_2|}2$ because there is no edge of type $U_2 U_2 W_2$. Since $v_0$ is not an isolated vertices, $v_0$ is adjacent to some vertex $u\in U$. Trivially $\deg(u) \leq \binom{|U|-1}{2} + (|U| - 1)|W|$. Thus
\begin{align*}
\deg(v_0) + \deg(u) &\le \binom{|U|-1}{2} + (|U| - 1)|W| + \binom{|U|}{2} - \binom{|U_2|}2 = (n-1)(|U| - 1) -  \binom{|U_2|}2.
\end{align*}
Since $|U| \geq 2s$ and $|M| < s-\varepsilon'' n$, it follows that $|U_2| =|U|-2|M| > 2\varepsilon'' n $. As a result,
\begin{eqnarray*}
\text{deg}(u)+\text{deg}(v_0)  \leq (n-1)(2s+\varepsilon' n-1)-\binom{2\varepsilon'' n}{2},
\end{eqnarray*}
which contradicts the condition that  $\text{deg}(u)+\text{deg}(v_0) > 2sn -\varepsilon n^2$ because $ \varepsilon  \ll \varepsilon'  \ll  \varepsilon'' $.
\end{proof}

\vskip.2cm

\begin{claim}\label{claim15}
Every vertex in $U$ is adjacent to one vertex in $W$.
\end{claim}
\begin{proof} To the contrary, assume that $u\in U$ is not adjacent to any vertex in $W$. Then
\[\text{deg}(u) \leq \binom{|U|-1}{2} < \binom{2s+\varepsilon'n}{2},\]
which contradicts the condition that $\text{deg}(u) > sn-\frac{1}{2}\varepsilon n^2$ because $\tau n < s \leq n/3 $ and $ \varepsilon  \ll \varepsilon' \ll\tau$.
\end{proof}

Fix $u_1 \ne u_2 \in U_2$ and $v_0 \in W_2$.
Trivially $\deg(w) \leq \binom{|U|}{2}$ for any vertex $w \in W$ and $\deg(u) \leq \binom{|U|-1}{2}+|W|(|U|-1)$ for any vertex $u \in U$. Furthermore, for any two distinct edges $e_1, e_2\in M$, we observe that at least one triple of type $UUW$ with one vertex in $e_1$, one vertex in $e_2$ and one vertex in $\{u_1,u_2,v_0\}$ is \emph{not} an edge by the choice of $M$.
By Claim \ref{claim16}, $|M| \geq s - \varepsilon''n$. Thus,
\begin{align*}
\deg(u_1)+ \deg(u_2)+ \deg(v_0)\leq 2\left( \binom{|U|-1}{2}+|W|(U|-1)\right) +\binom{|U|}{2}-\binom{s - \varepsilon''n}{2}.
\end{align*}

On the other hand, Claim \ref{claim15} implies that $u_i$ is adjacent to some vertex in $W$ for $i=1,2$. We know that $v_0$ is adjacent to some vertex in $U$. Therefore, $\deg(u_i) > \left(2sn-\varepsilon n^2\right)-\binom{|U|}{2}$ for $i=1,2$, and $\deg(v_0) > \left(2sn-\varepsilon n^2\right)-\left(\binom{|U|-1}{2}+|W|(|U|-1)\right)$. It follows that
\begin{align*}
\deg(u_1)+ \deg(u_2)+\deg(v_0)> 3\left(2sn-\varepsilon n^2\right) - 2\binom{|U|}{2}- \binom{|U|-1}{2}-|W|(|U|-1).
\end{align*}
The upper and lower bounds for $\deg(u_1)+ \deg(u_2)+\deg(v_0)$ together imply that
\begin{align*}
3\left( \binom{|U|-1}{2}+|W|(|U|-1)+\binom{|U|}{2}\right)-\binom{s-\varepsilon''n}{2} &> 3\left(2sn-\varepsilon n^2\right), \\
\text{or} \quad (|U| - 1)(n-1) - \frac13 \binom{s - \varepsilon''n}{2} &> 2sn-\varepsilon n^2,
\end{align*}
which is impossible because $|U| < 2s + \varepsilon' n$, $\tau n < s \leq n/3$, and  $\varepsilon  \ll \varepsilon'  \ll \varepsilon'' \ll\tau$.

\noindent{\bf Case 2. } $2s + \varepsilon' n \leq |U| \leq 3s$.

We consider the following two subcases.

\noindent{\bf Subcase 2.1. } No vertex in $U_2$ is adjacent to any vertex in $W_2$.

Fix $v_0\in W_2$. Then $\deg(v_0) \leq \binom{|U_1|}{2}=\binom{2|M|}{2}$.
Since $v_0$ is not an isolated vertex, $v_0$ is adjacent to some vertex $u_0\in U_1$. 
We know that $\deg(u_0)\leq \binom{|U|-1}{2}+(|U|-1)|W|-|U_2||W_2|$ because no vertex in $U_2$ is adjacent to any vertex in $W_2$.
Since $|W|= n - |U|$, $|U_2|=|U|-2|M|$ and $|W_2|=n-|U|-|M|$, we derive that
\begin{align*}
\sigma_2(H)\le \deg(v_0)+\deg(u_0) &\leq\binom{2|M|}{2}+\binom{|U|-1}{2} +(|U|-1)(n-|U|)-(|U|-2|M|)(n-|U|-|M|) \\
&\le (2n-|U|)|M|+\frac{|U|^2}{2}.
\end{align*}
Since $|M| < 3s-|U|$, it follows that 
\begin{align*}
\sigma_2(H) &< (2n-|U|)(3s-|U|)+\frac{|U|^2}{2} = 6sn-(3s+2n)|U|+\frac{3}{2} |U|^2.
\end{align*}
Note that the quadratic function $\frac 32 x^2 - (3s+2n)x$ is minimized at $x= s+ \frac 23 n$. Since 
$2s+\varepsilon'n \leq |U| \leq 3s\le s+ \frac 23 n$, we derive that
\begin{align*}
\sigma_2(H) &\le 6sn-(3s+2n)(2s+\varepsilon'n)+\frac{3}{2} (2s+\varepsilon'n)^2 \\
&= 2sn-2\varepsilon'n^2+3s \varepsilon' n+\frac{3}{2}\varepsilon'^2n^2   \le  2sn-\varepsilon' n^2 + \frac{3}{2}\varepsilon'^2n^2
\end{align*}
because $s \leq n/3$. Since $\varepsilon \ll \varepsilon' $, this contradicts the assumption that $\sigma_2(H)> 2sn - \varepsilon n$. 

\noindent{\bf Subcase 2.2. } Two vertices $u_0\in U_2$ and $v_0\in W_2$ are adjacent.

Let $M' = \{e \in M: \exists \, u'\in e, |N(v_0,u')\cap U_2| \geq 3\}$. Assume $\{u_1,u_2,v_1\} \in M'$ such that $u_1,u_2 \in U_1$, $v_1 \in W_1$ and $|N(v_0,u_1)\cap U_2|\geq 3$. We claim that
\begin{equation}
\label{eq:Nu0}
N(u_0,v_1)\cap U_2  = \emptyset.
\end{equation}
Indeed, if $\{u_0, v_1, u_3\} \in E(H)$ for some $u_3\in U_2$, then we can find $u_4\in U_2\setminus \{u_0, u_3\}$ such that $\{v_0, u_1, u_4\}\in E(H)$. Replacing $\{u_1,u_2,v_1\}$ by $\{u_0, v_1, u_3\}$ and $\{v_0, u_1, u_4\}$
gives a larger matching than $M$, a contradiction.

By the definition of $M'$, we have
\[
\deg(v_0) \leq \binom{|U_1|}{2}+2|M'||U_2|+2(|U_1|-2|M'|) =  \binom{|U_1|}{2} + 2|U_1| + |M'| (2|U_2|-4).
\]
By \eqref{eq:Nu0}, we have
\begin{align*}
\deg(u_0) &\leq \binom{|U|-1}{2}+|U_1||W|+(|U_2|-1)(|W_1|-|M'|)
\end{align*}
and consequently
\[
\deg(v_0)+\deg(u_0) \le \binom{|U_1|}{2} + \binom{|U|-1}{2}+|U_1|(|W|+2) + (|U_2|-1)|W_1| + |M'| ( |U_2|-3).
\]
Since $|M'| \leq |M| =|W_1|=\frac{|U_1|}{2}$, it follows that
\begin{align*}
 \deg(v_0)+\deg(u_0) & \le \binom{|U_1|}{2} + \binom{|U|-1}{2}+|U_1|(|W|+2) + (|U_2|-1) \frac{|U_1|}{2} + \frac{|U_1|}{2} (|U_2|-3)\\
& = \binom{|U|}{2}-\binom{|U_2|}{2}+\binom{|U|-1}{2}+|U_1||W| = \left(|U|-1\right)^2-\binom{|U_2|}{2}+2|M|\left(n-|U|\right).
\end{align*}
Since $|M| \leq 3s-|U|$ and $|U_2|=|U|-2|M| \ge 3|U|-6s$, we have 
\begin{eqnarray*}
\text{deg}(v_0)+\text{deg}(u_0) & \leq & (|U| - 1)^2-\binom{3|U|-6s}{2}+2(3s-|U|)\left(n-|U|\right)\\
& = & -\frac{3}{2}|U|^2+\left(12s-2n-\frac{1}{2} \right)|U|+6sn-18s^2-3s+1 \\
& \leq & -\frac{3}{2}|U|^2+\left(12s-2n\right)|U|+6sn-18s^2.
\end{eqnarray*}
Note that the quadratic function $-\frac 32 x^2 + (12s - 2n) x$ is maximized at $x= 4s - \frac 23 n$. Since $3s\ge |U| \ge 2s + \varepsilon'n \geq 4s-\frac{2}{3}n$, we have
\begin{eqnarray*}
\sigma_2(H)\le\text{deg}(v_0)+\text{deg}(u_0)&\leq &  -\frac{3}{2}(2s + \varepsilon'n)^2+\left(12s-2n\right)(2s + \varepsilon'n)+6sn-18s^2\\
& = & 2sn-2\varepsilon' n^2+6\varepsilon' sn-\frac{3}{2}\varepsilon'^2  n^2 \le  2sn- \frac{3}{2}{\varepsilon'}^2 n^2
\end{eqnarray*}
because $s \leq n/3$. Since $\varepsilon \ll \varepsilon' $, this contradicts the assumption that $\sigma_2(H)> 2sn - \varepsilon n$.  

\section{Proof of Theorem \ref{the2}}
Suppose $H$ is a $3$-graph of order $n$ without isolated vertex and $\sigma_2(H) > 2sn -\varepsilon n^2$.
Let $U = \{u\in V(H): \deg(u) > sn- \varepsilon n^2/2 \}$ and $W = V\setminus U$.  We know that no two vertices in $W$ are adjacent and $|U| \geq 2s$. Let $M$ be an optimal matching as in Definition~\ref{def:M}.
By Lemma \ref{lemma7}, such $M$ exists. 
Let $M_2 = M\setminus M_1$, $U_1 = V(M_1) \cap U$, $U_2 = V(M_2)$,  $U_3 = U\setminus V(M)$, $W_1 = V(M_1) \cap W$ and $W_2 =W\setminus W_1$.  
Since $M$ is optimal, no edge of $H$ is of type $W_2 U_3 U_3$ or $W_2 U_2 U_3$. In addition, for any $e\in M_1$, there are no two disjoint edges $e_1, e_2\in e\cup W_2\cup U_3$ such that $(e_1\cup e_2)\cap W_2 \ne \emptyset$.

Suppose to the contrary, that $|M| \leq s-1$.  We know that $|U_3|=|U|+|M_1|-3|M| \geq 3+|M_1|-(3s-|U|)\ge 3$. Let $u_1,u_2,u_3\in U_3$. Since $u_i \in U$ for $i=1,2,3,$ we have
\begin{align}\sum_{i=1}^{3}\deg(u_i) > 3sn-\frac{3}{2}\varepsilon n^2.\label{0001}\end{align}
On the other hand, if $u_1$ is adjacent to some $v_1\in W_2$, then
\begin{align}
\sum_{i=1}^{2}\deg(u_i)+\deg(v_1)\ge \sigma_2(H)+ \deg(u_2) > 3sn-\frac{3}{2}\varepsilon n^2.\label{211}
\end{align}


\vskip.2cm
\begin{claim}\label{claim17}
 For any two distinct edges $e_1$, $e_2$ from $M$, we have $\sum_{i=1}^3|L_{u_i}(e_1,e_2)| \leq 18$ and $\sum_{i=1}^2|L_{u_i}(e_1,e_2)|$ $+|L_{v_1}(e_1,e_2)| \leq 18.$
\end{claim}

\begin{proof} Let $H_1$ be the 3-partite subgraph of $H$ induced on three parts $e_1$, $e_2$, and $\{u_1,u_2,u_3\}$.
We observe that $H_1$ does not contain a perfect matching by the choice of $M$. By Lemma \ref{lemma1}, we obtain that $|E(H_1)| =\sum_{i=1}^3|L_{u_i}(e_1,e_2)| \leq 18$. Similarly, we have
$\sum_{i=1}^2|L_{u_i}(e_1,e_2)|+|L_{v_1}(e_1,e_2)| \leq 18.$
\end{proof}

We proceed in two cases.

\vskip.2cm
{\noindent \bf Case 1.} $|M_1| = 3s-|U|$.
\vskip.2cm
In this case, we have $|M_2|=|M|+|U|-3s$, $|U_3| = 3s-3|M|$ and $|W_2|=n-3s$. 

\begin{claim}\label{claim18}
 For any $e \in M_1$, we have

(i) $
\sum_{i=1}^2|L_{u_i}(e,U_3\cup W_2)|+|L_{v_1}(e,U_3\cup W_2)| \leq \max\{4|U_3|+7,3|U_3|+2|W_2|+5\}
$,where $v_1 \in W_2$;

(ii) $\sum_{i=1}^3|L_{u_i}(e,U_3 )| \leq 6|U_3|$.
\end{claim}

\vskip.2cm

\begin{proof} Assume $e =\{u_1', u_2', u_3'\} \in M_1$ with $u_1' \in W_1$ and $u_2', u_3'\in U_1$.

(i) Let $A=U_3$, $B=W_2$, and $E(G_i)=L_{u_i'}(U_3 \cup W_2)$ for $i=1, 2, 3$. 
By the choice of $M$, there are not two disjoint edges, one from $G_1$ and the other from $G_2$ or $G_3$; or  one from $G_2$ and the other from $G_3$, and at least one of them contains one vertex from $B$. Furthermore, it is easy to see that
\[\sum_{i=1}^2|L_{u_i}(e,U_3\cup W_2)|+|L_{v_1}(e,U_3\cup W_2)| = \sum_{i=1}^3\left(\sum_{j=1}^2 \deg_{G_i}(u_j)+ \deg_{G_i}(v_1)\right).\]
The desired inequality thus follows from Lemma \ref{Lemma4}.

(ii) For $i=1,2,3$, let $G_i$ be the graph obtained from $L_{u_i'}(U_3)$ after adding an isolated vertex $u^*$.  Then $|V(G_i)|=|U_3|+1 \geq 4$. By the choice of $M$, every edge of $G_1$ intersects every edge of $G_2$ and $G_3$. The desired inequality thus follows from Lemma \ref{lemmaa2}.
\end{proof}

\vskip.2cm

\begin{claim}\label{claim19} For any $e \in M_2$, we have

(i) $ \sum_{i=1}^3|L_{u_i}(e,U_3)| \leq 3(|U_3|+3)$;

(ii) $ \sum_{i=1}^2|L_{u_i}(e,U_3)| \leq  3(|U_3|+1)$.


\end{claim}
\vskip.2cm

\begin{proof} Assume $e =\{u_1',u_2',u_3'\} \in M_2$ with $u_1' ,u_2',u_3'\in U_2$.

(i) For $i=1,2,3$, let $G_i$ be the graph obtained from $L_{u_i'}(U_3)$ after adding two isolated vertices $u'$ and  $u''$.  Then $|V(G_i)|= |U_3|+2 \geq 5$. Since $M$ is optimal, the desired inequality follows from Lemma \ref{lemma3}.

(ii) For $i=1,2,3$, let $G_i$ be the graph obtained from $L_{u_i'}(U_3)$ after adding an isolated vertex $u^*$. Then $|V(G_i)| =|U_3|+1 \geq 4$.  Since $M$ is optimal, the desired inequality follows from Lemma \ref{lemmaa3}.
\end{proof}

\vskip.2cm
\begin{claim}\label{claim21} $s > n/3-\eta_1n$.
\end{claim}
\vskip.2cm
\begin{proof} Suppose $s \leq n/3-\eta_1n$. We first consider the case that $u_1,u_2,u_3$ are not adjacent to any vertex of $W_2$.

Following Claim \ref{claim17},  we have
\begin{align}
\sum_{i=1}^{3}\deg(u_i) &\leq 18\binom{|M|}{2}+9|M|+\sum_{i=1}^{3}|L_{u_i}(V(M_1),U_3)|+\sum_{i=1}^{3}|L_{u_i}(V(M_2),U_3)|. \label{2000}
\end{align}

Furthermore, by Claims \ref{claim18} (ii) and \ref{claim19} (i), we
obtain that
\begin{align*}
\sum_{i=1}^{3}\deg(u_i)
& \leq 18\binom{|M|}{2}+9|M|+6|M_1||U_3|+3|M_2|(|U_3|+3)\nonumber\\
 & = 18\binom{|M|}{2}+9|M|+6\left(3s-|U|\right)(3s-3|M|)+3(|M|+|U|-3s)(3s-3|M|+3) \nonumber\\
& = (9|U|-18s+9)|M|+ (3s- |U|)(9s- 9). 
\end{align*}

Since $|M| \leq s-1$, it follows that 
\begin{align*}
\sum_{i=1}^{3}\deg(u_i) & \leq (9|U|-18s+9)(s-1)+ (3s- |U|)(9s- 9)=  9s^2-9.
\end{align*}
Since $\tau n < s \leq n/3-\eta_1 n$ and $\eta_1 < \tau$, we know that
\begin{align}
3s^2-sn = s(3s- n) \leq  \max\left\{-\eta_1 n ( n - 3\eta_1 n), -\tau n (n - 3 \tau n)  \right\} = -\eta_1 n ( n - 3\eta_1 n). 
\label{008}
\end{align}
Consequently, $\sum_{i=1}^{3}\deg(u_i) < 9 s^2 \le 3sn - 3\eta_1 n ( n - 3\eta_1 n) $. Since
$\varepsilon \ll \eta_1$, this contradicts \eqref{0001}.


\medskip
Now we assume, without loss of generality, that $u_1$ is adjacent to $v_1$.
The choice of $M$ implies that  $L_v(e,U_3)=L_u(e,W_2)=\emptyset$ for any $v\in W_2$, $u\in U_3$ and $e\in M_2$. By Claim \ref{claim17}, we have
\begin{align}
\sum_{i=1}^2\deg(u_i)+\deg(v_1) &\leq 18\binom{|M|}{2}+9|M|+\sum_{i=1}^2|L_{u_i}(V(M_1),U_3\cup W_2)|+|L_{v_1}(V(M_1),U_3)|\nonumber\\
& \,\,\quad +\sum_{i=1}^2|L_{u_i}(V(M_2),U_3)|.\label{21000}
\end{align}
We know that $4|U_3|+7 \geq 3|U_3|+2|W_2|+5$ if and only if $|U_3| \geq 2|W_2|-2$.
If $|U_3| \geq 2|W_2|-2$, then by \eqref{21000}, Claim \ref{claim18} (i) and Claim \ref{claim19} (ii), we have
\begin{align*}
\sum_{i=1}^2\deg(u_i)+\deg(v_1)
& \leq 18\binom{|M|}{2}+9|M|+|M_1|(4|U_3|+7)+3|M_2|(|U_3|+1) \\
 & = 18\binom{|M|}{2}+9|M|+(3s-|U|)(4(3s-3|M|)+7)+3(|M|+|U|-3s)(3s-3|M|+1) \\
& = (3|U|+3)|M|-3s|U|-4|U|+9s^2+12s.
\end{align*}
Since $|M| \leq s-1$ and $|U| \geq 2s$, it follows that
\begin{align*}
\sum_{i=1}^2\deg(u_i)+\deg(v_1) & \leq (3|U|+3)(s-1)-3s|U|-4|U|+9s^2+12s\\
& =-7|U|+9s^2+15s-3 \leq 9s^2+s-3.
\end{align*}
Following \eqref{008}, we have $\sum_{i=1}^2\deg(u_i)+\deg(v_1) <  3sn - 3\eta_1 n ( n - 3\eta_1 n) + n/3 - 3$.
Since $\varepsilon \ll \eta_1$ and $n$ is sufficiently large, this contradicts \eqref{211}.

If $|U_3| <  2|W_2|-2$,  by \eqref{21000}, Claim \ref{claim18} (i) and Claim \ref{claim19} (ii), we have
\begin{align*}
\sum_{i=1}^2\deg(u_i)+\deg(v_1)
 &\leq  18\binom{|M|}{2}+9|M|+|M_1|\left(3|U_3|+2|W_2|+5\right)+3|M_2|(|U_3|+1) \\
 & = (9s+3)|M|+(-2n+6s-2)|U|+6sn-18s^2+6s.
\end{align*}
Since $|M| \leq s-1$ and $|U| \geq 2s$, it follows that
\begin{align*}
\sum_{i=1}^{2}\deg(u_i)+\deg(v_1) & \leq (9s+3)(s-1)+(-2n+6s-2)(2s)+6sn-18s^2+6s \\
& = 2sn+3s^2-4s-3.
\end{align*}
Applying \eqref{008}, we have $\sum_{i=1}^2\deg(u_i)+\deg(v_1) < 3sn - \eta_1 n ( n - 3\eta_1 n)$, which contradicts \eqref{211} because $\varepsilon \ll \eta_1$. 
\end{proof}

By Claim \ref{claim21}, we have $|W_2| = n-3s < 3\eta_1 n$. Let $H'= H[V\setminus W_2]$. We claim that $\sigma_2(H') > 2n^2/3-\eta_2 n^2$. Indeed, recall that $\text{deg}_H(u)+\text{deg}_H(v) \geq 2n^2/3-\varepsilon n^2$ for any two adjacent vertices $u$ and $v$ of $H'$. Since $|W_2| < 3\eta_1 n$ and $\varepsilon \ll \eta_1 \ll \eta_2$, it follows that
\[
\text{deg}_{H'}(u)+\text{deg}_{H'}(v) \geq 2n^2/3-\varepsilon n^2 -2|W_2|n > 2n^2/3-\eta_2 n^2.
\]
Since $\eta_2\ll 1$, we may apply Theorem \ref{theoremb6} and conclude that either $H'$ is a subgraph of $H_{3s,s}^2$ or $H'$ contains a perfect matching. In the former case, there is a partition of $V(H')$ into two sets $|T|=2s-1$ and $|S|=s+1$ such that for every vertex $u\in S$,
\[
\deg_{H'}(u)\le \binom{|T|}{2} = \binom{2s-1}{2} \le \binom{2n/3-1}{2} < \frac29 n^2.
\]
On the other hand, since $U\subseteq V(H')$ and $|U|\ge 2s$, there exists a vertex $u\in U\cap S$ such that
\begin{align*}
\deg_{H'}(u) &\geq \text{deg}_{H}(u)-|W_2|n \geq  sn-\frac{\varepsilon}{2} n^2-|W_2|n \\
	&\geq \left(\frac{n}{3}-\eta_1n \right)n-\frac{\varepsilon}{2} n^2- 3\eta_1 n^2 > \frac29 n^2,
\end{align*}
which is a contradiction.
Therefore $H'$ must contain a perfect matching, which is a matching of size $s$ in $H$.

\vskip.2cm

{\noindent \bf Case 2.} $|M_1| > 3s-|U|$.

The difference from Case 1 is that, for any edge $e \in M$, we cannot find two disjoint edges $e_1,e_2$ from $ e\cup U_3 \cup W_2$ -- otherwise we can replace $M$  by $M\setminus \{e\} \cup \{e_1,e_2\}$ contradicting the assumption that $M$ is an optimal matching.

Note that $|U_3|=|U|+|M_1|-3|M| \geq 3s+1 - 3|M| \geq 4$.
\begin{claim}\label{claim23} For any $e \in M$,
 $\sum_{i=1}^{3}|L_{u_i}(e,U_3 \cup W_2)| \leq 3(|U_3|+|W_2|+2)$.
\end{claim}

\begin{proof} Assume $e =\{u_1',u_2',u_3'\} \in M$. For  $i=1,2,3$, let $G_i$ be the graph obtained from $L_{u_i'}(U_3\cup W_2)$ after adding an isolated vertex $u^*$. Then $|V(G_i)| = |U_3|+|W_2|+1 \geq 5$. Since $H$ contains no two disjoint edges $e_1,e_2$ from $ e\cup U_3 \cup W_2$,  we know that for any $i\ne j$, every edge of $G_i$ intersects every edge of $G_j$. The desired inequality thus follows from Lemma \ref{lemma3}.
\end{proof}

\medskip




By Claims \ref{claim17} and \ref{claim23}, we obtain that
\begin{align}\label{eq7000}
\sum_{i=1}^{3}\deg(u_i) &\leq 18\binom{|M|}{2}+9|M|+\sum_{i=1}^{3}|L_{u_i}(V(M),U_3 \cup W_2)|\nonumber \\
& \leq 18\binom{|M|}{2}+9|M|+3|M|\left(|U_3|+|W_2|+2\right)\nonumber\\
& = (3n+6)|M| \le 3sn + 6s.
\end{align}

Let $W'=\{v \in W: \deg(v) \leq sn- s^2/2+\gamma' n^2\}$. If $|W'| \leq \gamma n$,  then we let $H': = H[V\setminus W' ]$.
By the definition of $W'$, $\deg_{H}(u) > sn- s^2/2+\gamma' n^2$ for every $u \in V(H')\cap W$. For any $u \in V(H')\cap U$, $\deg_{H}(u) >sn- \varepsilon n^2/2> sn- s^2/2 +\gamma' n^2$  because $s> \tau n$ and $\varepsilon \ll \gamma' \ll \tau$. Therefore every vertex $u\in V(H')$ satisfies
\[\deg_{H'}(u) \geq \deg_H(u)-n|W'| >  sn-\frac{s^2}{2} +\gamma' n^2-\gamma n^2 >  \binom{n-1}{2}-\binom{n-s}{2}+1,
\]
because $|W'| \leq \gamma n$, $\gamma \ll \gamma' $, and $n$ is sufficiently large. By Theorem \ref{Kuhn2}, $H'$ contains a matching of size $s$.

We thus assume that $|W'| > \gamma n$ for the rest of the proof.  If one of $u_1,u_2,u_3$ is adjacent to a vertex of $W'$, then
\[
\sum_{i=1}^{3}\deg(u_i) > 4\left(sn-\frac{\varepsilon}{2} n^2\right) -\left(sn-\frac{s^2}{2}+\gamma' n^2\right)=3sn+\frac{s^2}2 -2\varepsilon n^2-\gamma' n^2,
\]
which contradicts  \eqref{eq7000} because $s> \tau n$  is sufficiently large and $\varepsilon \ll \gamma' \ll  \tau$.

If none of $u_1,u_2,u_3$ is adjacent to a vertex of $W'$, then we distinguish the following two subcases.

{\noindent \bf Subcase 2.1.} $|W' \cap W_1| > \gamma n/2$.

Let  $M'=\{e \in M: e \cap W' \neq \emptyset\}$, thus $|M'| > \gamma n/2$.  Since $u_1,u_2,u_3$ are not adjacent to any vertex in $W' \cap W_1$, then for any distinct $e_1$, $e_2$ from $M'$, we have
\begin{align}\label{08}
\sum_{i=1}^3|L_{u_i}(e_1,e_2)| \leq 12.
\end{align}

By Claims \ref{claim17}, \ref{claim23} and \eqref{08}, we have
\begin{align*}
\sum_{i=1}^{3}\deg(u_i) &\leq \left(18\binom{|M|}{2}-6\binom{|M'|}{2}\right)+9|M|+3|M|\left(n-3|M|+2\right) \leq (3n+6)|M|-6\binom{|M'|}{2}.
\end{align*}
Since $|M'| > \gamma n/2$, it follows that
\begin{align*}
\sum_{i=1}^{3}\deg(u_i) & \leq (3n+6)(s-1)-6\binom{\gamma n/2}{2} ,
\end{align*}
which contradicts \eqref{0001} because $s\le n/3$ and $\varepsilon \ll \gamma$.

{\noindent \bf Subcase 2.2.} $|W' \cap W_1| \leq \gamma n/2$.

Since $|W'| > \gamma n$, we have $|W'  \cap W_2| > \gamma n/2$.  Let $W_2^* = W_2\setminus W'$. Then $W_2 \setminus W_2^*=W' \cap W_2$. By Claim \ref{claim23}, we obtain that $\sum_{i=1}^{3}|L_{u_i}(V(M),U_3 \cup W_2^*)| \leq 3|M|\left(|U_3|+|W_2^*|+2\right)$. Therefore,
\begin{align*}
\sum_{i=1}^{3}\deg(u_i) &\leq 18\binom{|M|}{2}+9|M|+\sum_{i=1}^{3}|L_{u_i}(V(M),U_3 \cup W_2^*)|\nonumber \\
& \leq 18\binom{|M|}{2}+9|M|+3|M|\left(|U_3|+|W_2^*|+2\right)\nonumber\\
& = 18\binom{|M|}{2}+9|M|+3|M|\left(|U_3|+|W_2|+2\right)-3|M||W_2\setminus W_2^*|\nonumber\\
& = \left(3n+6-\frac{3}{2}\gamma n\right)|M|,
\end{align*}
which contradicts \eqref{0001} because $|M| \leq s$, $\tau n< s$, and $\varepsilon \ll \gamma \ll \tau$.
This completes the proof of Theorem \ref{the2}.


\begin{thebibliography}{99}

\bibitem{Aha} R. Aharoni and D. Howard,
\newblock A rainbow r-partite version of the Erd\H{o}s-Ko-Rado theorem,
\newblock {\em Combin. Probab. Comput.} 26 (2017), 321--337.

\bibitem{Alon} N. Alon, P. Frankl, H. Huang, V. R\"{o}dl, A. Ruci\'{n}ski, and B. Sudakov. Large matchings in uniform hypergraphs and the conjectures of Erd\H{o}s and Samuels, \newblock {\em J. Combin. Theory Ser. A} 119 (2012), 1200--1215.








\bibitem{Erd65}
P.~Erd{\H{o}}s,
\newblock A problem on independent {$r$}-tuples,
\newblock {\em Ann. Univ. Sci. Budapest. E{\H{o}}tv{\H{o}}s Sect. Math.} 8 (1965), 93--95.

\bibitem{Fra}
P. Frankl, On the maximum number of edges in a hypergraph with given matching number,  \newblock {\em Discrete Appl. Math.} 216 (2017), 562--581.

\bibitem{Han} H. H\`{a}n, Y. Person, and M. Schacht, On perfect matchings in uniform hypergraphs with large minimum vertex degree, \newblock {\em SIAM J. Discrete Math.} 23 (2009), 732--748.


\bibitem{Han3} J. Han, Perfect matchings in hypergraphs and the Erd\H{o}s matching conjecture, \newblock {\em SIAM J. Discrete Math.} 30 (2016), 1351--1357.





\bibitem{Kha1} I. Khan, Perfect matching in 3-uniform hypergraphs with large vertex degree, \newblock {\em SIAM J. Discrete Math.} 27 (2013), 1021--1039.

\bibitem{Kha2} I. Khan, Perfect matchings in 4-uniform hypergraphs, \newblock {\em J. Combin. Theory Ser. B} 116 (2016), 333--366.

\bibitem{Kuhn1} D. K\"{u}hn and D. Osthus, Matchings in hypergraphs of large minimum degree, \newblock {\em J. Graph Theory} 51 (2006), 269--280.


\bibitem{Kuhn2} D. K\"{u}hn, D. Osthus and A. Treglown, Matchings in 3-uniform hypergaphs, \newblock {\em J. Combin. Theory Ser. B} 103 (2013), 291--305.


\bibitem{LuMi}
T. \L uczak and K. Mieczkowska, On Erd\H os' extremal problem on matchings in hypergraphs, \newblock {\em J. Combin. Theory, Ser. A} 124 (2014), 178--194.

\bibitem{Mar} K. Markstr\"{o}m and A. Ruci\'{n}ski, Perfect matchings (and Hamilton cycles) in hypergraphs with large degrees, \newblock {\em European J. Combin.} 32 (2011), 677--687.

\bibitem{Ore} O. Ore, Note on Hamilton circuits.  \newblock {\em Amer. Math. Monthly} 67 (1960), 55.

\bibitem{Pik} O. Pikhurko, Perfect matchings and $K^3_4$-tilings in hypergraphs of large codegree, \newblock {\em Graphs Combin.} 24 (2008), 391--404.

\bibitem{RoRu-s} V. R\"odl and A. Ruci\'nski, Dirac-type questions for hypergraphs -- a survey (or more problems for Endre to solve),
\emph{An Irregular Mind}, Bolyai Soc. Math. Studies~21 (2010), 561--590.

\bibitem{Rod2} V. R\"{o}dl, A. Ruci\'{n}ski, and E. Szemer\'{e}di, Perfect matchings in uniform hypergraphs with large minimum degree, \newblock {\em European J. Combin.} 27 (2006), 1333--1349.

\bibitem{Rod1} V. R\"{o}dl, A. Ruci\'{n}ski, and E. Szemer\'{e}di, An approximate Dirac-type theorem for $k$-uniform hypergraphs, \newblock {\em  Combinatorica,} 28 (2008), 229--260.

\bibitem{Rod3} V. R\"{o}dl, A. Ruci\'{n}ski, and E. Szemer\'{e}di, Perfect matchings in large uniform hypergraphs with large minimum collective degree, \newblock {\em J. Combin. Theory Ser. A} 116 (2009), 613--636.

\bibitem{Tang} Y. Tang and G. Yan, An approximate Ore-type result for tight hamilton cycles in uniform hypergraphs, \newblock {\em Discrete Math.} 340 (2017), 1528--1534.

\bibitem{TrZh12} A. Treglown and Y. Zhao, Exact minimum degree thresholds for perfect matchings in uniform hypergraphs,
\newblock {\em J. Combin. Theory Ser. A} 119 (2012), 1500--1522.

\bibitem{TrZh13}
A.~Treglown and Y.~Zhao,
\newblock Exact minimum degree thresholds for perfect matchings in uniform
  hypergraphs {II},
\newblock {\em J. Combin. Theory Ser. A} 120 (2013), 1463--1482.


\bibitem{TrZh16} A. Treglown and Y. Zhao, A note on perfect matchings in uniform hypergraphs,
\newblock {\em Electron. J. Combin.} 23 (2016), P1.16.



\bibitem{Yi1} Y. Zhang and M. Lu, Some Ore-type results for matching and perfect matching in $k$-uniform hypergraphs, \newblock {\em Acta. Math. Sin. -- English Ser.} 34 (2018) 1795-1803.

\bibitem{Yi2} Y. Zhang and M. Lu, $d$-matching in $3$-uniform hypergraphs, \newblock {\em Discrete Math.} 341 (2018), 748--758.

\bibitem{Yi3} Y. Zhang and M. Lu, Matching in $3$-uniform hypergraphs, submitted.

\bibitem{zhang} Y. Zhang, Y. Zhao and M. Lu, Vertex degree sums for perfect matchings in 3-uniform hypergraphs,
\newblock {\em Electron. J. Combin.} 25 (2018), P3.45.

\bibitem{Zhao} Y. Zhao,
\newblock Recent advances on dirac-type problems for hypergraphs.
\newblock In {\em Recent Trends in Combinatorics}, volume 159 of {\em the IMA
  Volumes in Mathematics and its Applications}. Springer, New York, 2016.

\end{thebibliography}
\end{document}